\documentclass[12pt]{amsart}
\usepackage[margin=2cm, bottom=2cm]{geometry}
\usepackage{amssymb}
\usepackage{amsmath}
\usepackage{amsthm}
\usepackage{amsfonts}
\usepackage{array}
\newcolumntype{P}[1]{>{\centering\arraybackslash}p{#1}}
\newcolumntype{M}[1]{>{\centering\arraybackslash}m{#1}}
\usepackage{hyperref}
\usepackage{mathtools}
\usepackage{float}
\restylefloat{table}
\usepackage{multirow}
\usepackage{bbm}
\usepackage{chngcntr}
\usepackage{listings}
\usepackage[hypcap]{caption}
\usepackage{caption}
\theoremstyle{plain}

\newtheorem{defss}[subsubsection]{Definition}

\newtheorem{theorem}[]{Theorem}
\newtheorem{prop}[subsection]{Proposition}

\newtheorem{lem}[subsection]{Lemma}

\newtheorem{coro}[subsection]{Corollary}

\newtheorem{remarks}[subsection]{Remark}

\newcommand{\mfg}{{\mathrm {\mathfrak{g}}}}

\newcommand{\bbz}{{\mathrm {\mathbb{Z}}}}

\newcommand{\bbk}{{\mathrm {\mathbbm{k}}}}

\newcommand{\n}{{\mathrm {\underline{n}}}}
\newcommand{\er}{{\mathrm {Er}}}
\newcommand{\al}{{\mathrm {\alpha}}}

\newcommand{\im}{{\mathrm {im}}}
\newcommand{\dims}{{\mathrm {dim}}}

\newcommand{\divs}{{\mathrm {div}}}
\newcommand{\ad}{{\mathrm {ad}}}

\newcommand{\Lie}{{\mathrm {Lie}}}

\newcommand{\mfsl}{{\mathrm {\mathfrak{sl}}}}
\newcommand{\Er}[1]{{{\mathrm{Er}(1;1)^{{(#1)}}}}}

\begin{document}
\lstset{language=Matlab}
\title[Ermolaev algebra is maximal in $F_4$]{The restricted Ermolaev algebra and $F_4$}
\thanks{}
\author{Thomas Purslow}
\date{}
\address{The University of Manchester, Oxford Road, M13 9PL, UK}
\email{thomas.purslow@manchester.ac.uk}
\begin{abstract} We investigate the simple Lie algebra of type $F_4$ over an algebraically closed field of characteristic $p=3$. In this article we show that the first Ermolaev algebra makes an appearance as a maximal subalgebra of $F_4$. To prove this we use old results from \cite{KU89} and \cite{ko95} about graded depth one simple Lie algebras over fields of characteristic $p=3$. \end{abstract}
\maketitle

\section*[Introduction]{Introduction}

Let $\bbk$ be an algebraically closed field of characteristic $p=3$ and $G$ an algebraic group of type $F_4$ such that $\mfg=\Lie(G)$. Over the complex numbers the classification of maximal subalgebras of the simple Lie algebras was achieved by Dynkin \cite{Dyn52}, and the process of extending this classification to the modular case is underway with \cite{HS14} and \cite{P15} for $p \ge 5$. The simple subalgebras of the exceptional Lie algebras are either of classical type or isomorphic to $W(1;1)$ \cite[Theorem 1.3]{HS14}.

We focus on \cite[Remark 1.4]{HS14} which produced an example of a $26$ dimensional simple subalgebra, $L$, in $\mfg$ that is neither classical nor $W(1;1)$. We confirm that $L$ is isomorphic to the restricted Ermolaev algebra and provide our own argument of maximality with the main result of this article.

\begin{theorem}\label{Ermo} Let $e:=e_{1000}+e_{0100}+e_{0001}+e_{0120}$ be a representative for the nilpotent orbit denoted $F_4(a_1)$. For $f:=f_{1232}$ we have $L:=\langle e,f \rangle \cong \Er{1}$ is a maximal subalgebra of $\mfg$.\end{theorem}

A consequence of working with algebraically closed fields of characteristic $p<5$ is the lack of Premet-Strade classification \cite{PSt08} presenting issues when it comes to identifying simple Lie algebras. We are extremely lucky that graded simple Lie algebras of depth one have been well studied with \cite{ko95} producing a recognition theorem and \cite{KU89} considering simple Lie algebras exhibiting depth one gradings with non-semisimple zero component.

\section*{Acknowledgments}
The author would like to thank his supervisor Professor Alexander Premet for his extremely helpful observations on $\Er{1}$ and many comments on earlier drafts and David Stewart for his very useful GAP advice. We must also acknowledge his former undergraduate student, Alex Kubiesa, who discovered a $26$ dimensional maximal subalgebra of $F_4$ that also appears to be $\Er{1}$. It is possible that this is conjugate to $L$ under the adjoint action of $G$. The author would also like to acknowledge the support of an EPSRC Doctoral Training Award for this research.

\section{Preliminaries}

\subsection{Notation}
Throughout let $G$ be an algebraic group of type $F_4$ with $\mfg=\Lie(G)$. We define as usual $\Pi=\{\al_i\}$ to be a basis of simple roots for the root system $\Phi$ of $G$. We express any root as a linear combination of such simple roots and using the ordering given in \cite{Bour02}. To shorten the notation further we specify the roots by giving the list of coefficients, for example the highest root in $F_4$ will be expressed as $\{2342\}$

This well-known construction allows us to define the simple Lie algebra $\mfg$ over an algebraically closed field $\bbk$ of positive characteristic using a Chevalley basis \cite{Che56}, taking basis elements indexed by the roots with basis $\{e_{\al}:\al \in \Phi\}$ and $\{h_{\al}: \al \in \Pi\}$. We write $e_{\al_1+\al_2}$ as $e_{1100}$ and for the negative roots use $f_{\al}$ to denote $e_{-\al}$.

\subsection{Nilpotent orbits}
We familiarise ourselves with nilpotent orbits and the classification of such orbits. For every nilpotent element $e \in \mfg$ we can form the orbit of $e$ under the adjoint action of $G$ on $\mfg$.

\begin{defss} A nilpotent element $e \in \mfg$ is said to be \emph{distinguished} if $C_G(e)^0$ is a unipotent group. \end{defss}

For every nilpotent element $e \in \mfg$ there is a Levi subgroup $L$ such that $e$ is distinguished in the Lie algebra of the dervied subgroup $L'$ of $L$. For example, we take a maximal torus $T$ of $C_G(e)$ and consider the Levi subgroup of $L=C_G(T)$. It follows from the structure of the restricted Lie algebra $\Lie(L)$ that $e \in \Lie(L')$.

In fields of characteristic zero or $p\gg0$ we use Jacobson-Morozov to associate an $\mfsl(2)$ triple to every nilpotent element. This observation is used to prove the Bala-Carter classification of nilpotent orbits for $p$ sufficiently large \cite{BaCar761} and \cite{BaCar762}. To obtain the same classification in good characteristic we use cocharacters to replace $\mfsl(2)$ triples and allow a classification to be found \cite{Pre03}. We discuss cocharacters with the main idea being the grading it admits on $\mfg$.

\begin{defss}Let $e$ be a nilpotent element in $\mfg$ such that $e$ is distinguished in some Levi subgroup $L$. A cocharacter $\tau: \bbk^{\ast} \rightarrow G$ is \emph{associated} to $e$ if both \[e \in \mfg(\tau,2)\,\,\text{and}\,\, \im(\tau) \subseteq [L,L]\]\end{defss}

Any cocharacter produces a $\bbz$-grading $\mfg=\bigoplus_i \mfg(\tau,i)$ such that $[\mfg(\tau,i),\mfg(\tau,j)]\subseteq \mfg(\tau,i+j)$. The classification of all nilpotent orbits in the exceptional Lie algebras is well known at this stage with the main reference \cite{LT11} which contains an incredible amount of information including tables containing orbit representatives, cocharacters and much more.

It must be noted that for bad primes many results on nilpotent orbits do not hold. The main concern for this article is the possible lack of associated cocharacters. It is possible we may only have $e \in \bigoplus_{i \ge 2}\mfg(\tau,i)$ for nilpotent elements in bad characteristic. We are fortunate for the work of \cite{VIG05} providing a list of nilpotent orbit representatives in bad characteristic with some errors corrected by \cite[Section 1.1]{S16}, and for \cite{hs85} producing a result regarding the number of new nilpotent orbits.

We study the nilpotent orbit with representative $e:=e_{1000}+e_{0100}+e_{0001}+e_{0120}$. It is easy to check using GAP that $\dims(\mfg_e)=6$ and so the orbit $\mathcal{O}(e)$ is subregular because there is only one nilpotent orbit of codimension $6$ in $F_4$ by \cite[Theorem 4]{hs85}. Hence we label the nilpotent orbit $F_4(a_1)$ as in the characteristic zero case. Hence in our case we can still use the associated cocharacter given \cite[pg81]{LT07} by \cite{CP13}.

\subsection{The restricted Ermolaev algebra}

The Ermolaev algebras are a class of simple Lie algebras different to the Cartan or classical type only appearing in algebraically closed fields of characteristic $p=3$, first constructed in \cite{er82}. We construct the restricted case using the description from \cite[Section 4.4]{Strade04} giving a nice grading on such a class of simple Lie algebras.

We have a map \[\divs: W(2;1) \rightarrow \mathcal{O}(2;1)\] such that $\divs(\sum f_i \,\partial_i)=\sum \partial_i(f_i)$ and for any $\al \in \bbk$ we have a $W(2;1)$-module denoted by $\mathcal{O}(2;1)_{(\al\,\divs)}$ obtained by taking $\mathcal{O}(2;1)$ under the action \[D\cdot f:=D(f)+\al \,\divs(D)f\] for all $D \in W(2;1)$ and $f \in \mathcal{O}(2;1)$.

The restricted \emph{Ermolaev} algebra, as a vector space, is $W(2;1)\oplus \mathcal{O}(2;1)_{(\divs)}$ which admits an automorphism of order two with $1$-eigenspace $W(2;1)$ and $(-1)$-eigenspace $\mathcal{O}(2;1)_{(\divs)}$. The Lie bracket is given by \[[f,g]:=(f\partial_2(g)-g\partial_2(f))\partial_1+(g\partial_1(f)-f\partial_1(g))\partial_2\] for all $f,g \in \mathcal{O}(2;1)$ with all other products defined canonically.

To obtain a simple Lie algebra from this we note that if $\al=1$ then $\mathcal{O}(2;1)_{(\al\,\divs)}$ has a submodule of codimension $1$ by \cite[Proposition 4.3.2, (1)]{Strade04} denoted by $\mathcal{O}'(2;1)_{(\divs)}$ and the derived subalgebra of $\er(1,1)$ is equal to \[W(2;1)\oplus \mathcal{O}'(2;1)_{(\divs)}\] This is a simple Lie algebra of dimension $26$ called the restricted \emph{Ermolaev} algebra.

We inherit a $\mathbb{Z}$-grading on $\er(1,1)$ from the usual grading of $W(2;1)$ with \[\er(1,1)_i:=W(2;1)_i\oplus \mathcal{O}(2;1)_{i+1}\] for all $i \ge -1$. From the grading we immediately see \[\er(1,1)_{-1}=\bbk\partial_1+\bbk\partial_2+\bbk1\] and $\er(1,1)_0$ of dimension $6$ with nilpotent radical. Taking the quotient of the zero component by its radical gives $\mfsl(2)$. 

\begin{remarks}\label{topco}By \cite[Proposition 4.3.2 (1)]{Strade04} we know that $\mathcal{O}'(2;1)_{(\divs)}$ is a proper submodule that does not contain ${x_1}^2{x_2}^2$ and so the top component of $\Er{1}$ is two dimensional. \end{remarks}

We also note that this construction can be generalised to all vectors $\underline{n} \in \mathbb{N}^2$ and consider \[\er(n_1,n_2):=W(2;\n) \oplus \mathcal{O}(2;\n)_{(\divs)}\] to produce simple Lie algebras of dimension $3^{n_1+n_2+1}-1$. This is the so called \emph{Ermolaev} series. We could attempt this for all $p>0$ but for $p \ne 3$ we lack the Jacobi identity, to see this consider \[J(x_1\partial_1,x_1,x_2)=[x_1\partial_1,[x_1,x_2]]+[x_2,[x_1\partial_1,x_1]]+[x_1,[x_2,x_1\partial_1]]=3(x_1\partial_1+x_2\partial_2)\]

\subsection{GAP calculations}
We frequently use \cite{GAP4} in this article, so we list the type of commands we use. We obtain the exceptional simple Lie algebras including $F_4$ along with a Chevalley basis with the following commands
\begin{lstlisting}
gap> g:=SimpleLieAlgebra(``F",4,GF(3));;
gap> b:=Basis(g);;
gap> e:=b[1]+b[2]+b[4]+b[10];;
\end{lstlisting}GAP does not produce $F_4$ in the same order as the usual Bourbaki ordering, it differs by some permutation on the roots. We find a full list of all the basis elements in \cite{de08} for all exceptional Lie algebras corresponding to their usual ordering. As an example, $e$ as above produces $e_{0100}+e_{1000}+e_{0120}+e_{0001}$.
We ask GAP to compute subalgebras on occasion, for example to obtain $\langle e,f \rangle$ and its dimension we enter
\begin{lstlisting}
gap> h:=Subalgebra(g,[e,f]);;
gap> Dimension(h);
26
\end{lstlisting}
The next tool we use checks the simplicity of a Lie algebra. This is the well-known MeatAxe package based on \cite{hr94}, which requires us to write the adjoint module as a collection of matrices. Then use the package to check whether the module is absolutely irreducible.
\begin{lstlisting}
gap> bh:=Basis(h);;
gap> Mats:=List(bh,x->AdjointMatrix(bh,x));;
gap> gm:=GModuleByMats(Mats,GF(3));;
gap> MTX.IsAbsolutelyIrreducible(gm);
true
\end{lstlisting}
This checks to see if our algebra is absolutely simple, we do this because we have defined the Lie algebras over finite fields so being simple over a finite field does not necessarily imply the algebra is simple over an algebraically closed field. We also generate vector spaces and check the dimension using
\begin{lstlisting}
gap> V:=VectorSpace(GF(3),[b[1],b[7]]);;
gap> Dimension(V);
2
\end{lstlisting}
\section{$\Er{1}$ is a subalgebra of $F_4$}
We prove the restricted Ermolaev algebra is a subalgebra of $\mfg$. Using some of what we discussed in the previous section we are able to obtain the following information using GAP.

\begin{enumerate}\item{The subalgebra defined in Theorem \ref{Ermo}, $L:=\langle e,f \rangle$, is simple of dimension 26,}\item{for $f':=f_{1222}-f_{1242}\in L$ the subalgebra $W:=\langle e,f' \rangle$ is simple of dimension 18.}\end{enumerate}

\begin{remarks}It has to be noted that our choice of $f$ seems to be chosen randomly and it is disappointing that we do not know the reason this element works. However, it is possible the fact $f \in \mfg(\tau,-10)$ is important and there could be a possible argument in the same spirit as \cite[Lemma 3.4, 3.5]{HS14} placing restrictions on the degrees of elements used to obtain simple subalgebras.\end{remarks}

We need to find the complementary module to $W(2;1)$ in the Ermolaev algebra and regrade $L$ to resemble the grading of $\Er{1}$ to conclude $L \cong \Er{1}$ by \cite[Theorem 1]{KU89}.

\begin{prop}\label{ersub}The subalgebra $L$ is isomorphic to $\Er{1}$.\end{prop}
\begin{proof} To begin we grade the subalgebra $L$ using the cocharacter from \cite[pg81]{LT11} to obtain the degree of each element of $L$ giving the following table

\begin{center}\begin{tabular}{|P{2cm}|P{0.8cm}|P{0.8cm}|P{0.8cm}|P{0.8cm}|P{0.8cm}|P{0.8cm}|P{0.8cm}|P{0.8cm}|P{0.8cm}|P{0.8cm}|P{0.8cm}|l}\hline $i$ &$-14$&$-12$&$-10$&$-8$&$-6$&$-4$&$-2$&$0$&$2$&$4$&$6$ \\ \hline $\dim(\mfg(\tau,i))$&$1$&$1$&$3$&$3$&$3$&$3$&$3$&$3$&$3$&$2$&$1$\\\hline\end{tabular}\end{center}To find the complementary module of $W$ in $L$ we attempt to find the element $1 \in \mathcal{O}'(2;1)_{(\divs)}$ as an element of $L$ and then try to construct an $8$ dimensional vector space, $V$, with the action of $W$ on such an element.

We consider $\ker(\ad\,e) \cap L(4)$, which is a one dimensional vector space with basis element $w:=e_{0111}-e_{1110}$ which we show is the element $1$ we are looking for. We compute the $\bbk$-span of all brackets $[x,w]$ for $x \in W$ to produce $V$ such that $L=W \oplus V$. The next step is to show $[V,V]$ generates $W$, we look at the span of $[u,v]$ for all $u,v \in V$. Using GAP to do this produces a simple Lie algebra of dimension $18$ which is easily seen to be $W$ if we look at the basis elements of both $W$ and $[V,V]$.

In the table below we provide the basis element for each one dimensional component $V(\tau,i)$. In the last column we give the new degree $d(i)$ for each component that will be used to regrade $L$.

\begin{center}\begin{tabular}{|M{1cm}|P{3.5cm}|M{1cm}|}\hline $i$ & Basis element of $V(\tau,i)$ &$d(i)$ \\ \hline\hline $4$&$e_{0111}-e_{1110}$&$-1$\\ \hline $2$&$e_{0011}-e_{0110}$&$0$\\ \hline $0$&$e_{0010}$&$1$ \\ \hline $-2$&$f_{0011}+f_{0110}$&$0$ \\ \hline $-4$&$f_{0111}+f_{1110}$&$1$\\ \hline $-6$&$f_{1111}$&$2$ \\ \hline $-8$&$f_{1231}$&$1$\\ \hline $-10$&$ f_{1232}$&$2$\\ \hline\end{tabular}\end{center}
These new degrees allow us to regrade the simple Lie algebra $L$ with $L_{-1}$ of dimension three with basis elements, \[\{e_{0111}-e_{1110}, e_{1121}+e_{0122}-e_{1220}, e_{0001}+e_{1000}+e_{0100}\}\] obtained by computing the elements of $V(4), [V(4),V(2)]$ and $[V(4),V(-2)]$.

We obtain $L_0$ by taking the obvious Lie brackets along with the elements of degree $0$ from $V$ to produce \[L_0=\langle V(2),V(-2),[V(4),V(0)], [V(4),V(-4)], [V(4),V(-8)], [V(2),V(-2)]\rangle\] This is a $6$ dimensional Lie algebra consisting of an $\mfsl(2)$ triple \[\{e_1,f_1,h_1\}\] with $e_1:=e_{0121}+e_{1120}, f_1:=f_{0121}+f_{1120}$ and $h_1:=h_{\al_1}+h_{\al_4}$ along with a $3$ dimensional radical with basis $\{V(2),V(-2),[V(-2),V(2)]\}$.

 This gives a depth one graded simple Lie algebra whose $L_0$ component contains a non-central nilpotent radical. Hence we can use \cite[Theorem 1]{KU89} along with the grading to see $L \cong \Er{1}$.
\end{proof}

\begin{coro}\label{witt}The subalgebra $W:=\langle e,f' \rangle$ is isomorphic to $W(2;1)$ \end{coro}\begin{proof}Every element of $W$ is obtained via the $\bbk$-span of Lie brackets $[u,v]$ from $V$ and so we obtain a grading on $W$ as follows.

For $W_{-1}$ we have a two dimensional vector space with basis elements $e_{1121}+e_{0122}-e_{1220}$ and $e_{0001}+e_{1000}+e_{0100}$. These are obtained taking the brackets $[V(4),V(2)]$ and $[V(4),V(-2)]$ respectively. We calculate the degree $0$ component in the same way to obtain \[W_0=\langle[V(4),V(0)], [V(4),V(-4)], [V(4),V(-8)], [V(2),V(-2)]\rangle\] which consists of an $\mfsl(2)$ triple \[\{e_1,f_1,h_1\}\] with $e_1:=e_{0121}+e_{1120}, f_1:=f_{0121}+f_{1120}$, $h_1:=h_{\al_1}+h_{\al_4}$ and central element $h_{\al_2}+h_{\al_4}$.

This gives a depth one graded simple Lie algebra with classical simple $W_0$ (modulo its centre). Using \cite[Theorem 1]{ko95} in combination with $\dims(W)=18$ tells us $W \cong W(2;1)$.\end{proof}

\section{The Ermolaev algebra is maximal}

In the final section we prove that $L$ is a maximal subalgebra of $F_4$ and in the process complete the proof of Theorem \ref{Ermo}. To do this we use well known results about the simple Lie algebra of type $F_4$. The starting point is the following lemma

\begin{lem}The Lie algebra $L:=\Er{1}$ is not self dual.\end{lem}\begin{proof}By construction we have $\dims(L_{-1})=3$ and $\dims(L_3)=2$ from Remark \ref{topco}, hence $L_{-1}\ncong (L_3)^{\ast}$. It follows by \cite[Lemma 4]{Pre85} that $L$ is not self dual. \end{proof}

We use this fact about the Ermolaev algebra to conclude maximality in $F_4$.

\begin{prop}The subalgebra $L$ is maximal in the exceptional Lie algebra of type $F_4$.\end{prop}
\begin{proof}$L$ is not self dual by the previous lemma, but simple Lie algebras of type $F_4$ are self dual and admit a non-degenerate invariant symmetric form denoted by $\kappa$, the so-called normalised Killing form defined explicitly \cite[pg 661]{CP13}. Since $\dims(L)=\frac{1}{2}\dims(\mfg)=26$ we know $L$ is a maximal totally isotropic subspace of $F_4$ with respect to $\kappa$.

Suppose it was not maximal and so there exists $M$ such that $L \subsetneq M \subsetneq \mfg$. The restriction of $\kappa$ to $M$ is non-zero with non-zero radical $R$ since $L$ is totally isotropic and $\dims(L) > \frac{1}{2}\dims(M)$. If $L \cap R=0$ then $L \oplus R$ becomes a maximal totally isotropic subalgebra containing $L$ which is a contradiction. Since $L$ is simple and $L \cap R \ne 0$ it must be the case $R \cap L=L$. In particular we conclude $R=L$ and $M \subseteq N_{\mfg}(L)=L$ proving $L$ is maximal. This completes the proof of Theorem \ref{Ermo}.\end{proof}

\bibliographystyle{amsalpha}\bibliography{g2}

\providecommand{\bysame}{\leavevmode\hbox to3em{\hrulefill}\thinspace}
\providecommand{\MR}{\relax\ifhmode\unskip\space\fi MR }
\providecommand{\MRhref}[2]{%
  \href{http://www.ams.org/mathscinet-getitem?mr=#1}{#2}
}
\providecommand{\href}[2]{#2}
\begin{thebibliography}{VAG05}

\bibitem[BC76a]{BaCar761}
P.~Bala and R.~W. Carter, \emph{Classes of unipotent elements in simple
  algebraic groups {I}}, Math. Proc. Cambridge Philos. Soc. \textbf{79} (1976),
  401--425.

\bibitem[BC76b]{BaCar762}
\bysame, \emph{Classes of unipotent elements in simple algebraic groups {II}},
  Math. Proc. Cambridge Philos. Soc. \textbf{80} (1976), 1--17.

\bibitem[Bou02]{Bour02}
N.~Bourbaki, \emph{Lie groups and {Lie} algebras. {Chapters} 4-6, {Elements} of
  {Mathematics}}, Springer-Verlag, Berlin, Berlin, 2002.

\bibitem[Che56]{Che56}
C.~Chevalley, \emph{Sur certains groupes simples}, Tohoku Math. J. \textbf{7}
  (1956), 14--66.

\bibitem[CP13]{CP13}
M.~C. Clarke and A.~Premet, \emph{The {H}esselink stratification of nullcones
  and base change}, Invent. Math. \textbf{191} (2013), no.~3, 631--669.

\bibitem[dG08]{de08}
W.~A. de~Graaf, \emph{Computing with nilpotent orbits in simple {L}ie algebras
  of exceptional type}, LMS J. Comput. Math. \textbf{11} (2008), 280--297.

\bibitem[Dyn52]{Dyn52}
E.~B. Dynkin, \emph{Semisimple subalgebras of semisimple {Lie} algebras}, Mat.
  Sb. \textbf{30} (1952), 349--462.

\bibitem[Erm82]{er82}
Yu.~B. Ermolaev, \emph{On a family of simple {L}ie algebras over a field of
  characteristic {$3$}}, Abstracts of the Fifth All-Union Symposium on Rings,
  Algebras and Modules, Novosibirsk (1982), 52--53.

\bibitem[GAP16]{GAP4}
The GAP~Group, \emph{{GAP -- Groups, Algorithms, and Programming, Version
  4.8.4}}, 2016.

\bibitem[HR94]{hr94}
D.~F. Holt and S.~Rees, \emph{Testing modules for irreducibility}, J. Austral.
  Math. Soc. Ser. A \textbf{57} (1994), no.~1, 1--16.

\bibitem[HS85]{hs85}
D.~F. Holt and N.~Spaltenstein, \emph{Nilpotent orbits of exceptional {L}ie
  algebras over algebraically closed fields of bad characteristic}, J. Austral.
  Math. Soc. Ser. A \textbf{38} (1985), no.~3, 330--350.

\bibitem[HS16]{HS14}
S.~Herpel and D.~I. Stewart, \emph{Maximal subalgebras of {C}artan type in the
  exceptional {L}ie algebras}, Selecta Math. (N.S.) \textbf{22} (2016), no.~2,
  765--799.

\bibitem[KO95]{ko95}
A.~I. Kostrikin and V.~V. Ostrik, \emph{On a recognition theorem for {L}ie
  algebras of characteristic {$3$}}, Mat. Sb. \textbf{186} (1995), no.~10,
  73--88.

\bibitem[Kuz89]{KU89}
M.~I. Kuznetsov, \emph{Classification of simple graded {L}ie algebras with
  nonsemisimple component {$L_0$}}, Mat. Sb. \textbf{180} (1989), no.~2,
  147--158, 304.

\bibitem[LT07]{LT07}
R.~Lawther and D.~Testerman, \emph{Centres of centralizers of unipotent
  elements in exceptional algebraic groups}, preprint 301 pp. (2007).

\bibitem[LT11]{LT11}
\bysame, \emph{Centres of centralizers of unipotent elements in exceptional
  algebraic groups}, Mem. Amer. Math. Soc. No. 210 \textbf{988} (2011).

\bibitem[Pre85]{Pre85}
A.~Premet, \emph{Algebraic groups associated with {Cartan} {Lie} $p$-algebras},
  Math. USSR Sbornik \textbf{50} (1985), 85--97.

\bibitem[Pre03]{Pre03}
\bysame, \emph{Nilpotent orbits in good characteristic and the
  {Kempf}-{Rousseau} theory}, J. Algebra \textbf{260} (2003), 338--366.

\bibitem[Pre15]{P15}
\bysame, \emph{A modular analogue of {M}orozov's theorem on maximal subalgebras
  of simple {L}ie algebras}, arXiv Preprint arXiv:1512.05614v4 [math.RA], 2015.

\bibitem[PS08]{PSt08}
A.~Premet and H.~Strade, \emph{Simple {Lie} algebras of small characteristic
  {VI.} {Completion} of the classification}, J. Algebra \textbf{320} (2008),
  3559--3604.

\bibitem[Ste15]{S16}
D.I. Stewart, \emph{On the minimal modules for exceptional {L}ie algebras,
  {J}ordan blocks and stabilisers}, arXiv:1508.02918v3 [math.RT], 2015.

\bibitem[Str04]{Strade04}
H.~Strade, \emph{Simple {Lie} {Algebras} over {Fields} of {Positive}
  {Characteristic}, {Volume} {I}: {Structure} {Theory}}, DeGruyter Expositions
  in Math., Vol.38, Berlin, 2004.

\bibitem[VAG05]{VIG05}
University of~Georgia VIGRE Algebra~Group, \emph{Varieties of nilpotent
  elements for simple {L}ie algebras. {II}. {B}ad primes}, J. Algebra
  \textbf{292} (2005), 65--99, The University of Georgia VIGRE Algebra Group:
  David J. Benson, Philip Bergonio, Brian D. Boe, Leonard Chastkofsky, Bobbe
  Cooper, G. Michael Guy, Jeremiah Hower, Markus Hunziker, Jo Jang Hyun,
  Jonathan Kujawa, Graham Matthews, Nadia Mazza, Daniel K. Nakano, Kenyon J.
  Platt and Caroline Wright.

\end{thebibliography}

\end{document}